\documentclass{amsart}
\usepackage{amsmath,amssymb}
\usepackage{graphicx,tikz,mathtools}
\newtheorem{theorem}{Theorem}[section]
\newtheorem{proposition}[theorem]{Proposition}

\newtheorem{lemma}[theorem]{Lemma}

\theoremstyle{definition}

\theoremstyle{remark}
\newtheorem{remark}[theorem]{Remark}

\newcommand{\ep}{\varepsilon}

\newcommand{\om}{\omega}

\newcommand{\De}{\Delta}

\def\CC{\mathbb{C}}

\def\RR{\mathbb{R}}

\def\ZZ{\mathbb{Z}}

\def\TT{\mathbb{T}}

\newcommand{\cL}{{\mathcal L}}

\newcommand{\cN}{{\mathcal N}}

\newcommand{\cV}{{\mathcal V}}

\newcommand{\pd}{\partial}
\newcommand\minus\backslash

\newcommand\lan\langle
\newcommand\ran\rangle

\DeclareRobustCommand{\Bint}
{\mathop{%
		\text{%
			\settowidth{\intwidth}{$\int$}%
			\makebox[0pt][l]{\makebox[\intwidth]{$-$}}%
			$\int$}}}

\newcommand{\supp}{\operatorname{supp}}

\DeclareMathOperator\Div{div}

\newcommand\intT{\int_{\TT^2}}

\renewcommand\leq\leqslant
\renewcommand\geq\geqslant
\newlength{\intwidth}

\numberwithin{equation}{section}

\newcommand\esp{C^\infty_{\mathrm{div}}(\TT^d,\RR^d)}
\newcommand\espR{C^\infty_{\mathrm{div}}(\RR^d,\RR^d)}
\newcommand\espdos{C^\infty_{\mathrm{div}}(\TT^2,\RR^2)}
\newcommand\bv{\bar v}
\newcommand\bp{\bar p}

\begin{document}

\title[Quasi-periodic solutions to the Euler equations]{Quasi-periodic solutions to the incompressible Euler equations in dimensions two and higher}

 \author{Alberto Enciso}
 \address{Instituto de Ciencias Matem\'aticas, Consejo Superior de
   Investigaciones Cient\'\i ficas, 28049 Madrid, Spain}
 \email{aenciso@icmat.es}

 \author{Daniel Peralta-Salas}
 \address{Instituto de Ciencias Matem\'aticas, Consejo Superior de
   Investigaciones Cient\'\i ficas, 28049 Madrid, Spain}
 \email{dperalta@icmat.es}

 \author{Francisco Torres de Lizaur}
 \address{Departamento de An\'alisis Matem\'atico \& IMUS, Universidad de Sevilla, 41012 Sevilla, Spain}
 \email{ftorres2@us.es}

%%    General info
%\subjclass[2010]{35B38, 58J05, 58K45}
%\date{\today}
%
%\keywords{ }
%
\begin{abstract}
Building on the work of Crouseilles and Faou on the 2D case, we construct $C^\infty$ quasi-periodic solutions to the incompressible Euler
equations with periodic boundary conditions in dimension~$3$ and in
any even dimension. These solutions are genuinely high-dimensional,
which is particularly interesting because there are extremely few
examples of high-dimensional initial data for which global solutions
are known to exist. These quasi-periodic solutions can be engineered
so that they are dense on tori of arbitrary dimension embedded in the space of solenoidal vector fields. Furthermore,
in the two-dimensional case we show that quasi-periodic solutions are
dense in the phase space of the Euler equations. More precisely, for any integer $N\geq 1$ we prove that any $L^q$~initial stream function can be
approximated in~$L^q$ (strongly when $1\leq q< \infty$ and weak-$*$ when $q=\infty$) by smooth initial data whose solutions are dense on $N$-dimensional tori.
\end{abstract}

\maketitle

\section{Introduction}

As a rule of thumb, a Hamiltonian PDE defined on a compact manifold typically has many periodic and quasi-periodic (in time) solutions. Heuristically, this is because the equations are time-reversible, so dissipative phenomena cannot occur, and because a certain ``recurrent'' behavior can be expected from the fact that the spatial variable ranges over a bounded set. In contrast, when the spatial variable takes values in a noncompact manifold, such as~$\RR^d$, dispersive phenomena typically enter the picture, and the existence of quasi-periodic solutions for nonlinear PDEs becomes rare.

In the case of periodic solutions to nonlinear PDEs, the first results were obtained in the late 1970s using variational methods~\cite{Rabi,Brezis}, but these techniques do not carry over to the case of quasi-periodic solutions due to the appearance of small divisors. Quasi-periodic solutions to nonlinear wave and Schr\"odinger equations were first obtained in the late eighties, first in the one-dimensional case and later in higher dimensions. Building on the pioneering works of Kuksin~\cite{Ku}, Wayne~\cite{Wayne}, Craig~\cite{CW}, Bourgain~\cite{B1,B2}, P\"oschel~\cite{Pos,KP} and others, the field has grown into a vast area of research~\cite{CY,BamGe,GY,EK,BP,BBP,BB,HHM}. Typical techniques used nowadays are normal forms, Newton--Nash--Moser quadratic iterations, paradifferential calculus and Lindstedt series. For a nice introduction to this subject and an up to date bibliography, see~\cite{B} and references therein.

In the context of incompressible fluid mechanics, several important results about the existence of quasi-periodic solutions have been obtained~\cite{IP,IPT,Bplus,BM,BMon,BHM}, particularly in the context of the water wave equations (see also~\cite{EJ19} for the construction of quasi-periodic solutions of the Euler equations in $\RR^2$ using measure-valued vorticities). Given that proofs of the existence of (KAM-type) quasi-periodic solutions are often long and involved, specially in spatial dimensions higher than one, a beautiful and surprisingly easy result is Crouseilles and Faou's proof~\cite{CF} of the existence of quasi-periodic solutions (not of KAM-type) for the incompressible Euler equations,
\begin{equation}\label{E.Euler}
	\pd_t u + u\cdot\nabla u +\nabla p =0\,,\qquad \Div u=0\,,
\end{equation}
on the 2-dimensional torus~$\TT^2$, where $\TT^d:=(\RR/2\pi\ZZ)^d$. The proof is carried out in the vorticity-stream formulation of the 2D Euler equations,
\begin{equation}\label{E.stream}
	\pd_t\om=\nabla^\perp\psi \cdot\nabla\om\,,\qquad \Delta \psi = \om\,.
\end{equation}
A key fact of the construction is that, in two dimensions,
$$u=\nabla^\perp\psi:=(\partial_{x_2}\psi,-\partial_{x_1}\psi)$$
is a stationary solution of compact support whenever~$\psi$ is a radial, compactly supported function. Crouseilles and Faou's quasi-periodic solutions are constructed as a linear combination of localized traveling profiles, which are ``glued'' together using a carefully chosen global steady state.

Our first objective in this note is to extend the result of Crouseilles and Faou about quasi-periodic solutions to the Euler equations in dimensions higher than~2. That this should be feasible at least in dimension~3 is suggested by the recently established existence of smooth 3D steady Euler flows with compact support~\cite{Gavrilov,Constantin}. To fix the notation, let us start by recalling that a smooth solution to the Euler equations on~$\TT^d$ is called {\em quasi-periodic}\/ if there exist some $N\geq1$, an irrational constant vector $\nu\in\RR^N$ (i.e., $\nu\cdot k\neq 0$ for all $k\in\mathbb Z^N\backslash\{0\}$) and a smooth embedding $U: \TT^N\to \esp$, which may depend on $\nu$, in terms of which the velocity field can be written as
\[
u(t,\cdot)= U(\theta_0+\nu t)
\]
for some point $\theta_0\in\TT^N$. Here and in what follows, $\esp$ denotes the space of smooth divergence-free vector fields on the $d$-dimensional torus.

Of course, we want to construct solutions with a nontrivial dependence on all~$d$ variables; in fact, a feature that makes this problem particularly interesting in dimensions~$3$ and higher is that there are extremely few examples of genuinely high-dimensional initial data for which global solutions are known to exist (see however the recent result~\cite{Pausader}). We will therefore look for solutions where $U$ is {\em non-symmetric}\/ in the sense that for all $\theta\in\TT^N$ the divergence-free vector field $U(\theta)$ is not invariant under any 1-parameter group of translations on $\TT^d$. This condition ensures that the solutions really depend on all~$d$ coordinates, so they cannot be obtained from solutions to the Euler equations in a lower dimension. It is also convenient to use the notation $N=0$ for the degenerate case in which the torus is replaced by a set consisting of a single element, which corresponds to the case where $u(t,\cdot )=U(\theta_0)$ is a stationary solution.

We are now ready to state our first main result, which ensures that there are quasi-periodic solutions of the $d$-dimensional Euler equations which faithfully represent any linear flow on an $N$-dimensional torus. In dimension~2, this is the result proven by Crouseilles and Faou.

\begin{theorem}\label{T.torus}
	Assume that the dimension~$d$ is either $3$ or even. For every $N\geq1$ and every~$\nu\in\RR^N\backslash\{0\}$, there exists a non-symmetric embedding $U:\TT^N\to \esp$ and a family of initial data $u_\theta\in\esp$, $\theta\in\TT^N$, such that the corresponding solutions to the incompressible Euler equations are $u_\theta(t,\cdot)=U(\theta+\nu t)$.
\end{theorem}

\begin{remark}\label{R.Dioph}
We are not assuming that the frequency vector~$\nu$ is irrational.
If such a condition holds, the solution $u_\theta(t,\cdot)$ is dense on the embedded $N$-dimensional torus $U(\TT^N)\subset \esp$ for any $\theta\in\TT^N$.
\end{remark}

Theorem~\ref{T.torus} shows that there is a nontrivial embedding of any linear flow on an $N$-dimensional torus into the incompressible Euler equations on~$\TT^d$. By a theorem of Elgindi, Hu and Sverak~\cite{EHS}, recently extended by Kishimoto and Yoneda~\cite{KY}, at least when $d=2$ or~$3$, this $N$-dimensional torus cannot be contained in a subspace spanned by~$\CC^d$-valued trigonometric polynomials of bounded degree. A recent result of one of the authors~\cite{Paco} ensures that any flow on a compact $N$-dimensional manifold can be embedded, up to an arbitrarily small error, into the incompressible Euler equations on some compact Riemannian manifold of high dimension. This theorem utilizes a previous result of Tao~\cite{Tao} about the embedability of norm-preserving quadratic ODEs into the Euler equations on some compact Riemannian manifold.

The proof of Theorem~\ref{T.torus} follows the general strategy of Crouseilles and Faou. One cannot use the vorticity-stream formulation in dimensions higher than two, but a similar construction works directly in the velocity formulation. The solutions we construct are also suitable gluings of localized traveling profiles, so they are not of KAM type (that is, no small divisors arise). Smooth stationary Euler flows on~$\RR^d$ with compact support again play a key role in the construction. When $d=3$, solutions supported on axisymmetric tori have been recently constructed by Gavrilov; when $d$ is even, we present an elementary construction of smooth stationary solutions supported on balls.

Concerning the 2D Euler equations, it is worth mentioning that, although one can prove the result without any explicit reference to the stream function~$\psi$, one can use this formulation to prove a somewhat surprising result about the density of quasi-periodic solutions. A well known result of Yudovich~\cite[Chapter~8]{MB} shows that the 2D Euler equations are globally well posed provided that the initial vorticity~$\om_0$ is in~$L^\infty(\TT^2)$. Our approximation theorem, which is the second main result of this note, is stated in terms of the initial stream function $\psi_0:=\De^{-1}\om_0$. We do not know if a similar result holds in higher dimensions.
%\begin{theorem}\label{T.dim2}
%	Consider any function $\psi_0\in L^q(\TT^2)$, with $1\leq q\leq \infty$. Given any $N\geq0$ and any $\nu\in\RR^N$, there exist sequences of scalar functions $\psi^n_0\in  C^\infty(\TT^2)$, of  constants $c^n>0$ and of non-symmetric embeddings $U^n:\TT^N\to\espdos$ such that:
%\begin{enumerate}
%\item The solution to the incompressible Euler equations with initial datum $u^n_0=\nabla^\perp \psi^n_0$ is $u^n(t,\cdot)= U^n(c^n \nu t)$.
%\item The sequence $\psi^n_0$ converges weak-$*$ to~$\psi_0$ in $L^q(\TT^2)$.
%\end{enumerate}
%\end{theorem}
%
%\begin{remark}
%An analogous density result holds with strong $L^q$ convergence of the sequence $\psi^n_0$ to~$\psi_0$, provided that $1\leq q <2$.
%\end{remark}

\begin{theorem}\label{T.dim2}
	Consider any function $\psi_0\in L^q(\TT^2)$, with $1\leq q\leq  \infty$. Given any $N\geq0$ and any $\nu\in\RR^N\backslash\{0\}$, there exist sequences of scalar functions $\psi^n_0\in  C^\infty(\TT^2)$, of  constants $c^n>0$ and of non-symmetric embeddings $U^n:\TT^N\to\espdos$ such that:
\begin{enumerate}
\item The solution to the incompressible Euler equations with initial datum $u^n_0=\nabla^\perp \psi^n_0$ is $u^n(t,\cdot)= U^n(c^n \nu t)$.
\item The sequence $\psi^n_0$ converges to~$\psi_0$ in $L^q(\TT^2)$, strongly if $q<\infty$ and weak-$*$ if $q=\infty$.
\end{enumerate}
\end{theorem}
%\begin{remark}\label{R.qinf}
%If we assume $q=\infty$, an analogous density result holds with weak-$*$ convergence of the sequence $\psi^n_0$ to~$\psi_0$ (cf Remarks \ref{R.qinf2} and \ref{R.qinf3}).
%\end{remark}

Note, in particular, that the functions $u^n(t,\cdot)$ can be dense on an $N$-dimensional torus embedded in~$\espdos$ provided that $\nu$ is irrational (or the functions can also be stationary solutions, in the particular case $N=0$).

\section{Construction of quasi-periodic solutions: proof of Theorem~\ref{T.torus}}
\label{S.QP}

To construct quasi-periodic solutions on~$\TT^d$, we follow the strategy of Crouseilles and Faou, with the caveat that all computations must be carried out directly using the velocity field. The starting observation is that there are nontrivial $C^\infty$ stationary Euler flows of compact support in dimension~$d$:

\begin{proposition}\label{P.stat}
	If $d=3$ or $d$ is even, there are smooth compactly supported solutions to the stationary Euler equations on~$\RR^d$.
\end{proposition}

\begin{proof}
If $d=3$, this is a celebrated result of Gavrilov~\cite{Gavrilov,Constantin}. Let us now assume that $d$ is even. A short computation shows that the vector field on~$\RR^d$
	\[
	\widetilde u (x):= (x_2,-x_1,\dots, x_d,-x_{d-1})
	\]
	is divergence-free and satisfies the equation
	\[
	\widetilde u\cdot\nabla \widetilde u +\nabla \bp=0
	\]
	with $\widetilde p(x):=\frac12|x|^2$. Furthermore, $\widetilde u\cdot\nabla \widetilde p=0$. Therefore,
	\begin{equation}\label{E.tildeu}
			u(x):= f(|x|) \,\widetilde u(x)
	\end{equation}
	is a smooth stationary Euler flow with compact support
	for any even function $f\in C^\infty_c(\RR)$, and the corresponding pressure function is
	\[
	p(x):=  \int_0^{|x|} r\,f( r)^2\, d r\,.
	\]
	The proposition then follows.
	\end{proof}
	
	\begin{remark}\label{R.dim2}
	In dimension~2, \eqref{E.tildeu} is just the stationary solution associated with the radial stream function $\psi(x):=\int_0^{|x|} r\,f(r)\, dr$.	
	\end{remark}

Let $v\in\espR$ be a compactly supported solution to the stationary Euler equations, and let $p_v$ be its associated pressure function.
Up to a rescaling and translation of Euclidean coordinates, and adding a constant to~$p_v$ if necessary, we can assume that the supports of $v$ and~$p_v$ are contained in a ball centered at the origin of small radius~$\ep$. Let us now fix an integer $1\leq m\leq d-1$. Writing points on the $d$-dimensional torus as
\[
\TT^d\ni x=(x',x'')\in \TT^m\times\TT^{d-m}\,,
\]
and picking points $y^j\in\TT^m$ and vectors $\bar\nu^j\in\RR^{d-m}$ with $1\leq j\leq J$, one can define a time-dependent vector field on~$\TT^d$ as
\begin{equation}\label{E.u}
	u(t,x):= \sum_{j=1}^J \bv_j(t,x)+w(x)\,.
\end{equation}
Here
\[
\bv_j(t,x):=\bv(x'-y^j,x''-\bar\nu^jt)
\]
and
\[
\bv(x):= \sum_{k\in\ZZ^d} v(x+2\pi k)\in\esp
\]
is the periodic extension of $v$, which satisfies the stationary Euler equations on $\TT^d$ with pressure
\begin{equation}\label{E.pv}
p_{\bv}(x):= \sum_{k\in\ZZ^d} p_v(x+2\pi k)\in C^\infty(\TT^d)\,.
\end{equation}
As $\ep$ is small, only one term of the sum is nonzero for each~$x\in\RR^d$. Also, note that, as $\supp \bv_j(t,\cdot)\subset B'(y^j;\ep)\times\TT^{d-m}$, where $B'(y;r)$ denotes the ball in~$\TT^m$ of center~$y$ and radius~$r$,  the supports of the fields $\bv_j(t,\cdot)$ and $\bv_k(t',\cdot)$, $j\neq k$, are disjoint for all $t,t'$.

With an analogous splitting $\RR^d=\RR^m\times\RR^{d-m}$, it is  apparent that the vector field
\[
w(x):=(0,F(x'))\in \esp
\]
also satisfies the stationary Euler equations for any $F\in C^\infty(\TT^m,\RR^{d-m})$:
\begin{equation}\label{E.w}
w\cdot\nabla w=0\,,\qquad \Div w=0\,.	
\end{equation}

We now want to ensure that the vector field~$u$ defined by~\eqref{E.u} is a time-dependent solution of the Euler equations. For this, we will make some {assumptions} about the location of the points $y^j$ and about the structure of the function~$F$. To this end, we introduce the notation $\rho(z,y)$ for the distance between two points $y,z\in\TT^m$. Of course, if~$y$ is fixed and the distance between~$y$ and~$z$ is smaller than~$\pi$, the function~$\rho(\cdot,y)$ is smooth but at $z=y$ and can be written as
\begin{equation}\label{E.rho}
	\rho(z,y)=|y-z|\,,
\end{equation}
where (with some abuse of notation) we also denote by $y,z$ the unique representatives of the points in a certain $m$-dimensional cube of side $2\pi$.

We are now ready to make our assumptions precise:
\begin{enumerate}
	\item For all $ j\neq k$, $\rho(y^j,y^k)>4\ep$. 	
	\item The function $F$ is
	\[
	F(x'):=\sum_{j=1}^J \bar\nu^j \,\chi(\rho(x',y^j))\,,
	\]
	where
	$\chi(r)\in C^\infty(\RR)$ is some function which equals~1 if $|r|<\ep$ and~0 if $|r|>2\ep$.
	In particular, the vector field~$w$ defined above is locally constant, and equal to $(0,\bar\nu^j)$, on the support of $\bv_j$, for all $1\leq j\leq J$.
\end{enumerate}
These assumptions are clearly consistent provided that $\ep\ll 1/J$.
	
Using these assumptions and Equations~\eqref{E.u}-\eqref{E.w}, one can then compute:
\begin{align*}
	\pd_t u&= -\sum_{j=1}^J \bar\nu^j\cdot\nabla''\bv_j\,,\\
	u\cdot \nabla u&= \sum_{j,k=1}^J \bv_j\cdot \nabla\bv_k+ \sum_{j=1}^J (\bv_j\cdot \nabla w+ w\cdot\nabla \bv_j) + w\cdot\nabla w\\
	&=-\nabla\sum_{j=1}^J p_{\bv}(x'-y^j,x''-\bar\nu^jt)+ \sum_{j=1}^J \bar\nu^j\cdot\nabla''\bv_j\,.
\end{align*}
To pass to the last line we have used Equation~\eqref{E.w}, that $\bv_j$ and $\bv_k$ have disjoint supports and satisfy the equation
\[
\bv_j \cdot \nabla \bv_j=-\nabla p_{\bv}(x'-y^j,x''-\bar\nu^jt)\,,
\]
and that $w=(0,\bar\nu^j)$ on the support of $\bv_j$.

As $u$ is obviously divergence-free, it then follows from the previous computation that $u$ satisfies the Euler equations with pressure
\[
p(t,x):=\sum_{j=1}^J p_{\bv}(x'-y^j,x''-\nu^jt)\,.
\]

To construct the embedding of $\TT^N$, assume without any loss of generality that $N\leq J(d-m)$ and take a linear embedding $\cN: \TT^N\to\TT^{J(d-m)}$ which we write as
\[
\cN(\theta)=(\cN^1(\theta),\dots, \cN^J(\theta))
\]
with $\cN^j:\TT^N\to\TT^{d-m}$ (being linear means that each factor $\cN^j$ lifts to a linear map at the universal cover). Let us denote the differential of $\cN$ at $\theta=0$ as $\bar\cN: \RR^N\to\RR^{J(d-m)}$. Then we can take the desired family of (quasi-periodic) solutions to be $U(\theta+\nu t)$, $\theta\in\TT^N$, with $U:\TT^N\to\esp$ given by
\[
U(\theta+\nu t)(x):=\sum_{j=1}^J \bv(x'-y^j,x''-\cN^j(\theta)-\bar\cN^j(\nu)t) + \sum_{j=1}^J (0,\bar\cN^j(\nu)\,\chi(\rho(x',y^j))) \,,
\]
which correspond to initial data
\[
u_\theta(x):=U(\theta)=\sum_{j=1}^J \bv(x'-y^j,x''-\cN^j(\theta)) + \sum_{j=1}^J (0,\bar\cN^j(\nu)\,\chi(\rho(x',y^j)))\,.
\]

%To construct the embedding of $\TT^N$, assume without any loss of generality that $N\leq J(d-m)$ and take a linear embedding $\bar\cN: \RR^N\to\RR^{J(d-m)}$ that is the lift to the universal cover of an embedding $\cN: \TT^N\to\TT^{J(d-m)}$, which we write as
%\[
%\cN(\theta)=(\cN^1(\theta),\dots, \cN^J(\theta))
%\]
%with $\cN^j:\TT^N\to\TT^{d-m}$. Then we can take the desired family of (quasi-periodic) solutions to be $U(\theta+\nu t)$, $\theta\in\TT^N$, with $U:\TT^N\to\esp$ given by
%\[
%U(\theta+\nu t)(x):=\sum_{j=1}^J \bv(x'-y^j,x''-\cN^j(\theta)-\bar\cN^j(\nu)t) + \sum_{j=1}^J (0,\bar\cN^j(\nu)\,\chi(\rho(x',y^j))) \,,
%\]
%which correspond to initial data
%\[
%u_\theta(x):=U(\theta)=\sum_{j=1}^J \bv(x'-y^j,x''-\cN^j(\theta)) + \sum_{j=1}^J (0,\bar\cN^j(\nu)\,\chi(\rho(x',y^j)))\,.
%\]

\begin{remark}
In the particular case that $m=d-1$ and $J=N$, we can take the embedding $\cN:\TT^N\to\TT^N$ to be the identity, and the family of (quasi-periodic) solutions is given by
\[
U(\theta+\nu t)(x):=\sum_{j=1}^N \bv(x'-y^j,x''-\theta^j-\nu^jt) + \sum_{j=1}^N (0,\nu^j\,\chi(\rho(x',y^j))) \,,
\]
where $\theta=(\theta^1\cdots,\theta^N)\in\TT^N$.
\end{remark}

\begin{remark}
To prove a result analogous to Theorem~\ref{T.torus} in odd dimensions higher than~$3$, one just needs to establish the existence of $C^\infty$ compactly supported steady states. Probably a suitable generalization of  Gavrilov's solutions can be constructed on any odd dimension. If this is true, then the theorem would hold in all dimensions.
\end{remark}

%\section{Weak density of quasi-periodic solutions: proof of Theorem~\ref{T.dim2}}
%\label{S.dim2}

\section{Density of quasi-periodic solutions: proof of Theorem~\ref{T.dim2}}
\label{S.dim2}

The proof hinges on the following elementary lemma. To state the result concisely, let us say that a function $\psi$ on~$\TT^2$ is {\em locally radial}\/ if its support is a union of (pairwise disjoint) finitely many closed disks
$\overline{B(y^l,r^l)}$,
and that on each of these disks the function is radial:
\[
\psi(x)= f^l(|x-y^l|)
\]
for all $x\in B(y^l,r^l)$ and some function $f^l$. All along this section we describe points on $\TT^2$ using Euclidean coordinates $(x_1,x_2)\in (\RR/2\pi\ZZ)^2$.

%\begin{lemma}\label{L.locrad}
%	Let
%	\[
%	\cV:=\big\{x\in\TT^2: x_1\in\{s^1,\cdots, s^N\} \textcolor{red}{\subset [0, 2\pi)}\big \}
%	\]
%	be a finite collection of \textcolor{red}{N} vertical lines. There exists a sequence of locally radial functions $\{\phi^n\}\subset C^\infty_c(\TT^2\backslash\cV)$ such that
%	\begin{equation}\label{E.lemma}
%			\intT \phi^n g\, dx \to \intT\psi_0 g\, dx
%	\end{equation}
%	as $n\to\infty$, for all $g\in L^{q'}(\TT^2)$.
%\end{lemma}

\begin{lemma}\label{L.locrad}
	Let
	\[
	\cV:=\big\{x\in\TT^2: x_1\in\{s^1,\cdots, s^N\} \big \}
	\]
	be a finite collection of N vertical lines. If $q<\infty$, there exists a sequence of locally radial functions $\{\phi^n\}\subset C^\infty_c(\TT^2\backslash\cV)$ such that
	\begin{equation}\label{E.lemma}
			\|\phi^{n}-\psi_{0}\|_{L^{q}(\TT^{2})} \rightarrow 0
	\end{equation}
	as $n\to\infty$. If $q=\infty$, then $\phi^{n}\to \psi_{0}$ weak-$*$ in $L^\infty$ instead, that is,
	\begin{equation}\label{E.lemma2}
		\intT \phi^n g\, dx \to \intT\psi_0 g\, dx
\end{equation}
as $n\to\infty$, for all $g\in L^{1}(\TT^2)$.
\end{lemma}

\begin{proof}
Let us first assume that $1\leq q<\infty$. Since $C^\infty$ functions are dense in $L^q$, $1\leq q<\infty$ (in the $L^q$ norm), we can assume that $\psi_0\in C^\infty(\TT^2)$ without loss of generality.

%Analogously, as $q'\in (1,\infty]$, the same density argument for $q'<\infty$ allows us to assume that $g\in L^\infty(\TT^2)$.
	
	For each $n>1$, one can pick a finite collection of disjoint closed balls
	\[
	B_{nl}:= \overline{B(y_{nl};r_{nl})}\subset\TT^2\backslash\cV\,,
	\]
	where $y_{nl}$ are points in~$\TT^2$ and where, for each~$n$, $l$ ranges from 1 to some integer~$L_n$, such that:
	\begin{enumerate}
		\item The radii do not depend on $l$ and are bounded as $r_{nl}\leq 1/n$.
		\item The union of the balls covers almost the whole torus in the sense that
		\[
		\left|\TT^2\backslash \bigcup_{l=1}^{L_n} B_{nl}\right| \leq\frac1n\,.
		\]
	\end{enumerate}
	Obviously, the area of each ball $B_{nl}$ does not depend on $l$, so we shall use the notation $A_n:=|B_{nl}|$.

	Consider now the step-like functions
	\[
	\Phi^n(x):= \sum_{l=1}^{L_n}1_{B_{nl}}(x)\Bint_{B_{nl}}\psi_0\,,
	\]
	where $\Bint_{B}F:=|B|^{-1}\int_BF$ denotes the average of a function $F$ over a set $B$. As we can safely assume that $\psi_0\in C^\infty(\TT^2)$, it is easy to see that
	\begin{align*}
\intT |\psi_0-\Phi^n|^{q}&=\sum_{l=1}^{L_n} \int_{B_{nl}} \bigg|\psi_0(x)-\Bint_{B_{nl}}\psi_0\bigg|^{q}\, dx + \int_{\TT^2\backslash \bigcup_{l=1}^{L_n} B_{nl}}|\psi_0(x)|^{q}\,dx\\  &\leq \frac{2^{q} \|\nabla\psi_0\|^{q}_{L^\infty(\TT^2)}L_nA_n}{n^{q}}+\frac{\|\psi_0\|^{q}_{L^{\infty}(\TT^2)}}{n}  \\
&\leq 2^{q} \frac{\|\psi_0\|^{q}_{C^1(\TT^2)}}{n}\,,
	\end{align*}
which converges to 0 as $n\to\infty$. To obtain the first estimate we have used the upper bound for $r_{nl}$, property~(ii) above and the fact that, by the mean value theorem, for any point $x$ on the ball $B_{nl}$ we can write the bound
\[
\bigg|\psi_{0}(x)-\Bint_{B_{nl}}\psi_0\bigg| \leq \frac{2}{n} \|\nabla \psi_{0}\| _{L^{\infty}(B_{nl})} \,.
\]
To pass to the third line we simply used that, by construction, $L_nA_n\lesssim 1$, and $(a^q+b^q) < (a+b)^{q}$ for positive numbers $a, b$.
	
Now, a simple argument allows us to ``smooth out'' the $L^\infty$ functions $\Phi^n$. Indeed, let $\{H_m\}_{m=1}^\infty$ be a sequence of smooth even functions on the real line, equal to~1 on $[-\frac12,\frac12]$, whose support is contained in $[-1,1]$, and converging to the characteristic function of the interval $[-1,1]$ strongly in $L^q(\RR)$. Then
	\begin{equation}\label{E.L1}
			H_m\bigg(\frac{\rho(x,y_{nl})}{r_{nl}}\bigg)\to 1_{B_{nl}}(x)
	\end{equation}
strongly in~$L^q$ as $m\to\infty$, where we are denoting by $\rho(x,y)$ the distance between two points $x,y\in\TT^2$. Therefore, using the previous estimate for the sequence $\Phi^n$, we easily infer that there is a subsequence $\{m_n\}$ such that the smooth locally radial functions
	\begin{equation}\label{E.phin}
			\phi^n(x):= \sum_{l=1}^{L_n}H_{m_n}\bigg(\frac{\rho(x,y_{nl})}{r_{nl}}\bigg)\Bint_{B_{nl}}\psi_0
	\end{equation}
	satisfy~\eqref{E.lemma}, as we wanted to prove. Notice that, by construction, $\phi^n$ is supported on $\bigcup_{l=1}^{L_n} B_{nl}$. This completes the proof of the lemma when $q<\infty$.

Let us now consider the case $q=\infty$ and establish~\eqref{E.lemma2}. The proof goes essentially as above. Since smooth functions are dense in~$L^1(\TT^2)$, we can safely assume that $g\in C^\infty(\TT^2)$. We can then write the estimate
\begin{align*}
&\intT \psi_0 g\, dx - \sum_{l=1}^{L_n}\int_{B_{nl}} \psi_0(x)\,dx  \Bint_{B_{nl}}g = \intT \Big(g(x)-\sum_{l=1}^{L_n} 1_{B_{nl}}(x)\Bint_{B_{nl}}g\Big)\psi_0(x)\,dx\\
 &\leq \sum_{l=1}^{L_n}\int_{B_{nl}}\Big|g(x)-\Bint_{B_{nl}}g\Big ||\psi_0(x)|\,dx+\int_{\TT^2\backslash \bigcup_{l=1}^{L_n} B_{nl}}|g(x)||\psi_0(x)|\,dx\\
 &\leq C\|\psi_0\|_{L^\infty(\TT^2)}\|g\|_{C^1(\TT^2)}\Big(\frac{L_nA_n}{n}+\frac1n\Big)\to 0
\end{align*}
as $n\to\infty$ because $L_nA_n\lesssim 1$. Accordingly, we conclude that
\[
\sum_{l=1}^{L_n}\int_{B_{nl}} \psi_0(x)\,dx  \Bint_{B_{nl}}g \to \intT \psi_0 g\, dx
\]
as $n\to\infty$. Since moreover
	\[
	\Bint_{B_{nl}} H_m\bigg(\frac{\rho(x,y_{nl})}{r_{nl}}\bigg)\, g(x)\,dx\to \Bint_{B_{nl}}g
	\]
as $m\to\infty$ by~\eqref{E.L1}, we conclude that there exists a subsequence $\{m_n\}$ such that the smooth locally radial function~\eqref{E.phin} satisfies~\eqref{E.lemma2}:
\begin{align*}
\left|\intT (\psi_0-\phi^n)g\right|&\leq\left|\intT \psi_0 g\, dx - \sum_{l=1}^{L_n}\int_{B_{nl}} \psi_0(x)\,dx  \Bint_{B_{nl}}g\right|\\
&+\sum_{l=1}^{L_n}\left|\Bint_{B_{nl}} H_{m_n}\bigg(\frac{\rho(x,y_{nl})}{r_{nl}}\bigg)\, g(x)\,dx-\Bint_{B_{nl}}g\right|\cdot\left|\int_{B_{nl}}\psi_0(x)\,dx\right|\\
&\to 0\,,
\end{align*}
and the claim follows.
\end{proof}

%\begin{remark}
%Arguing as in the first part of the proof of the lemma when $q<\infty$, it is easy to obtain the following bound
%\[
%\|\psi_0-\Phi^n\|_{L^q(\TT^2)}\leq \frac{C\|\psi_0\|_{C^1(\TT^2)} A_n^{1/q-1}}{n}\,.
%\]
%Therefore, taking the collection of closed balls $B_{nl}$ so that $r_{nl}\simeq n^{-1}$ and hence $A_n\simeq n^{-2}$, we infer that
%\[
%\|\psi_0-\Phi^n\|_{L^q(\TT^2)} \to 0
%\]
%as $n\to\infty$ provided that $1\leq q<2$. Defining smooth functions $\phi^n$ as in Equation~\eqref{E.phin}, and arguing as before, it follows the strong density in $L^q$, $1\leq q<2$, of locally radial functions.
%\end{remark}

We are now ready to present the proof of Theorem~\ref{T.dim2}. Let us start with the case $q<\infty$. We consider the collection of vertical lines
\[
	\cV:=\big\{\{\tfrac{2\pi j}N\}\times\TT: 1\leq j\leq N \big \}\,.
	\]
	Lemma~\ref{L.locrad} and Equation~\eqref{E.phin} show that a sequence of $C^\infty$ locally radial functions of the form
	\[
	\phi^n(x):= \sum_{l=1}^{L_n}a_{nl}\,H_{m_n}\bigg(\frac{\rho(x,y_{nl})}{r_{nl}}\bigg)\,,
	\]
	with
	\[
	a_{nl}:=\Bint_{B_{nl}}\psi_0\,,
	\]
	converges strongly to~$\psi_0$ in~$L^q$ for $q \in [1, \infty)$ as $n\to\infty$, where the closed balls~$\{B_{nl}\}_{l=1}^{L_n}$ have the same properties as in the proof of the lemma. For each $1\leq j\leq N$, we denote by
	\[
	\{B_{nl}: l\in\cL_{nj}\}
	\]
	the subset of closed balls located in the (open) region between the vertical lines $x_1= 2\pi j/N$ and $x_1=2\pi(j+1)/N$ (of course, we identify $2\pi(N+1)/N$ with $2\pi/N$). Recall that, by construction, there exists some large integer $M_n$ such that the distance between $B_{nl}$ and~$\cV$ is at least $2/M_n$ for all $l\leq L_n$.

	Now pick any numbers $\nu^j$, with $1\leq j\leq N$. For each~$1\leq j\leq N$ and each positive integer $M\gg N$, let $\chi_{j,M}\in C^\infty(\RR,[0,1])$ be a $2\pi$-periodic function such that
	\[
	\chi_{j,M}(s)=\begin{cases}
		s & \text{if } s\in [\frac{2\pi j}N+\frac2M,\frac{2\pi(j+1)}N-\frac2M]\,,\\
		0 & \text{if } s\in [0,\frac{2\pi j}N+\frac1M]\cup [\frac{2\pi(j+1)}N-\frac1M,2\pi]\,,
	\end{cases}
	\]
for $1\leq j\leq N-1$, and
\[
	\chi_{N,M}(s)=\begin{cases}
		s & \text{if } s\in [\frac2M,\frac{2\pi}N-\frac2M]\,,\\
		0 & \text{if } s\in [0,\frac1M]\cup [\frac{2\pi}N-\frac1M,2\pi]\,.
	\end{cases}
	\]
Obviously, $\chi_{j,M}$ and $\chi_{j',M}$ have disjoint supports if $j\neq j'$. Now we define
	\[
	F_n(s):= -\sum_{j=1}^N \nu^j\, \chi_{j,M_n}(s)\in C^\infty(\TT)\,.
	\]
	
	Finally we set, for $c\in\RR$, the smooth function
	\[
	\psi^{n,c}(t,x):= \sum_{j=1}^N \sum_{l\in\cL_{n,j}} a_{nl}\,H_{m_n}\bigg(\frac{\rho((x_1,x_2-c\nu^jt),y_{nl})}{r_{nl}}\bigg)+ cF_n(x_1)\,.
	\]	
It is straightforward to check that
\[
\nabla^\perp F_n=(0,\nu^j)
\]	
on the support of $H_{m_n}\bigg(\frac{\rho((x_1,x_2-c\nu^jt),y_{nl})}{r_{nl}}\bigg)$ for all $l\in\cL_{nj}$.

A minor modification of the proof of Theorem~\ref{T.torus} shows that the field
$$v^n(t,x):=\nabla^\perp \psi^{n,c}(t,x)$$
is a solution to the incompressible Euler equations for any positive integer~$n$ and all real $c$ (one can also follow directly the proof in~\cite{CF}).
Since $\phi^n\to\psi_0$ strongly in~$L^q$ as $n\to\infty$ and $\psi^{n,c}(0,\cdot)\to \phi^n$ strongly in~$L^q$ as $c\to0$, the theorem follows by taking
	\[
	\psi_0^n(x):=\psi^{n,c^n}(0,x)\,,
	\]
	where $\{c^n\}$ is a sequence of positive numbers that tends to zero fast enough. Indeed, the corresponding embeddings of $\TT^N$ are
	\[
	U^n(\theta):= \nabla^\perp \Big(\sum_{j=1}^N \sum_{l\in\cL_{nj}} a_{nl}\,H_{m_n}\bigg(\frac{\rho((x_1,x_2-\theta^j),y_{nl})}{r_{nl}}\bigg)+ c^nF_n(x_1)\Big)\,,
	\]
with $\theta=(\theta^1,\cdots,\theta^N)\in\TT^N$, and clearly the velocity field is
\[
u^n(t,\cdot)=U^n(c^n\nu t)=\nabla^\perp \Big(\sum_{j=1}^N \sum_{l\in\cL_{nj}} a_{nl}\,H_{m_n}\bigg(\frac{\rho((x_1,x_2-c^n\nu^jt),y_{nl})}{r_{nl}}\bigg)+ c^nF_n(x_1)\Big)\,,
\]
corresponding to the initial datum $u_0^n:=\nabla^\perp \psi_0^n$.
This completes the proof of the theorem.

When $q=\infty$, the argument remains unchanged, the only difference being that the application of Lemma~\ref{L.locrad} leads to a sequence converging to~$\psi_0$ in the $L^\infty$ weak-$*$ topology rather than in norm.

\section*{Acknowledgements}

The authors would like to thank \'Angel Castro for valuable discussions. This work has received funding from the European Research Council (ERC) under the European Union's Horizon 2020 research and innovation program through the Consolidator Grant agreement~862342 (A.E.) and the Marie Curie Fellowship 101063565 (F.T.L.). It is partially supported by the grants CEX2019-000904-S, RED2018-102650-T and PID2019-106715GB GB-C21 (D.P.-S.) funded by MCIN/ AEI/10.13039/501100011033, and a Fields Ontario Postdoctoral Fellowship (F.T.L.) financed by the NSERC grant RGPIN-2019-05209. D.P.-S. and F.T.L. also acknowledge partial support from the grant ``Computational, dynamical and geometrical complexity in fluid dynamics'', Ayudas Fundaci\'on BBVA a Proyectos de Investigaci\'on Cient\'ifica 2021.

\bibliographystyle{amsplain}

\end{document}